\documentclass[twoside,11pt,reqno]{amsart}

\usepackage{amsmath,amsthm,amssymb,amstext,amsfonts,amscd,mathabx}

\usepackage{graphicx,color, url}

\usepackage{multirow}

\newtheorem{theorem}{Theorem}

\numberwithin{equation}{section}

\begin{document}

\title[]{On some integral transforms of Coulomb functions related to three-dimensional proper Lorentz group}

\author[{\bf  I. A. Shilin}]{\bf I. A. Shilin}

\address{I. A. Shilin:  Department of Higher Mathematics, National Research University MPEI,
Krasnokazarmennaya 14, Moscow 111250, Russia}
 \email{shilinia@mpei.ru}

 \address{I. A. Shilin: Department of Algebra, Moscow State Pedagogical University,
Malaya Pirogovskaya 1, Moscow 119991, Russia}
 \email{ilyashilin@li.ru}

\bigskip



\begin{abstract} Considering the relationship between two bases in
representation space of the three-dimensional proper Lorentz
group, we derive some formulas with integrals involving Coulomb
wave functions, which can be considered as Fourier, Mellin,
$K$-Bessel, Hankel and Mehler-Fock transforms of these functions.
\end{abstract}

\thanks{$^*$ Corresponding author}

\maketitle

\section{Introduction}

As usually, let $\mathbb{R}$ and $\mathbb{C}$ be the sets of real
and complex numbers, respectively. In addition, throughout this
paper we use the definition $\mathbb{R}_a$ for the ray
$(a;+\infty)$.

Let us recall that the Coulomb (wave) functions
$F_\sigma(\rho,\lambda)$ and $H^\pm_\sigma(\rho,\lambda)$ are
functions belonging to the kernel of the Coulomb differential
operator
$$\mathfrak{d}:=\frac{{\rm d}^2}{{\rm
d}\lambda^2}+1-\frac{2\rho}\lambda-\frac{\sigma(\sigma+1)}{\lambda^2},$$
where $\lambda\in\mathbb{R}^0$, $\rho\in\mathbb{R}$ (Sommerfeld
parameter), and $\sigma$ is non-negative integer (angular momentum
quantum number). These functions are defined by formulas \cite{ek}
\begin{gather}F_\sigma(\rho,\lambda)=2^{-\sigma-1}C_\sigma(\rho)(\mp{\bf
i})^{\sigma+1}\,M_{\pm{\bf i}\rho,\sigma+\frac12}(\pm2{\bf
i}\lambda),\label{mitterand}\\H^\pm_\sigma(\rho,\lambda)=(\mp{\bf
i})^\sigma\exp\left(\frac{\pi\rho}2\pm{\bf
i}c_\sigma(\rho)\right)\,W_{\mp{\bf
i}\rho,\sigma+\frac12}(\mp2{\bf
i}\lambda),\label{thatcher}\end{gather} where $M_{\mu,\nu}(z)$ and
$W_{\mu,\nu}(z)$ are Whittaker functions of the first and second
kind, respectively, and the normalizing constant (Gamow factor)
$C_\sigma(\rho)$ and Coulomb phase shift $c_\sigma(\rho)$ are
defined as follow:
\begin{gather*}C_\sigma(\rho)=2^\sigma\exp\left(-\frac{\pi\rho}2\right)\left[\Gamma(2\sigma+2)\right]^{-1}
\left|\Gamma(\sigma+1+{\bf
i}\rho)\right|,\\c_\sigma(\rho)=\arg\Gamma(\sigma+1+{\bf
i}\rho).\end{gather*} They can be considered for complex values of
$\lambda(\ne0)$, $\rho$ and $\sigma$ \cite{dyf,g}. The definitions
\eqref{mitterand} and \eqref{thatcher} are correct since the
choice of upper or lower signs in \eqref{mitterand} and
\eqref{thatcher} isn't important in view of identity
${}_1F_1(a;b;z)=\exp(z)\,{}_1F_1(b-a;b;-z)$ (known as Kummer's
transformation) for confluent hypergeometric function ${}_1F_1$
which determines the both Whittaker functions.

We recall also that $F_\sigma(\rho,\lambda)$ and
$H^\pm_\sigma(\rho,\lambda)$ are regular and irregular
(respectively) solutions of the Coulomb wave equation
$\mathfrak{d}[y]=0$ and connected by the identity
\begin{equation}\label{de-gaulle}F_\sigma(\rho,\lambda)=\pm\,\,{\rm imaginary\,\, part\,\,
of\,\, }\,H^\pm_\sigma(\rho,\lambda).\end{equation} Below we also
use another function belonging to ${\rm Ker}\,\mathfrak{d}$:
$$G_\sigma(\rho,\lambda)=\,{\rm real\,\, part\,\, of\,\,
}\,H^\pm_\sigma(\rho,\lambda).$$ Since defect of $\mathfrak{d}$ is
equal to 2, the linearly independent functions
$F_\sigma(\rho,\lambda)$ and $G_\sigma(\rho,\lambda)$ form a basis
in ${\rm Ker}\,\mathfrak{d}$. Another basis consists of
$H^+_\sigma(\rho,\lambda)$ and $H^-_\sigma(\rho,\lambda)$. In
\cite{c} the author considered also two other bases in ${\rm
Ker}\,\mathfrak{d}$ consisting of functions, also named Coulomb
functions and introduced by Hartree in \cite{h} and H. and B.
Jeffreys in \cite{jj}.

\section{Representation space, its bases, and functionals ${\sf F}_1$ and ${\sf F}_2$}

We recall that the three-dimensional Lorentz group is the subgroup
of matrices $(g_{ij})$ in $GL(3,\mathbb{R})$ satisfying the
equalities
$g_{i1}^2-g_{i2}^2-g_{i3}^2=(-1)^{E\left(\frac{i}2\right)}$ for
$i\in\{1,2,3\}$, where $E(n)$ denotes the entire part of an
integer $n$. In this paper we consider its intersection $G$ with
$SL(3,\mathbb{R})$, calling $G$ the proper Lorentz group.

Let $\sigma\in\mathbb{C}$ and $T$ be the representation of $G$ in
the linear space $\mathfrak{D}$ consisting of $\sigma$-homogeneous
and infinitely differentiable functions defined on the cone
$\Lambda:\,x_1^2-x_2^2-x_3^2=0$ acting according to rule
$T(g)[f(x)]=f(g^{-1}x)$. We recall that the functions $x^\mu_\pm$
on $\mathbb{R}$, which generate the generalized functions
$(x^\mu_\pm,f)$ \cite{gsh}, are defined as follow: $x^\mu_\pm$ is
equal to $|x|^\mu$ for $x\in\mathbb{R}^\pm$ and coincides with
zero function otherwise.
 In this paper we deal with the bases \cite{vsh}
  $$B_1=\left\{f_\lambda(x)=(x_1+x_2)^\sigma\,\exp\frac{\lambda x_3}{x_1+x_2}\,\mid\,\lambda\in\mathbb{R}\right\}$$ and
   $$B_2=\left\{f_{\rho,\pm}(x)=(x_2)_\pm^{\sigma-{\bf i}\rho}\,(x_1+x_3)^{{\bf i}\rho}\,\mid\,\rho\in\mathbb{R}\right\}.$$

Below we use two bilinear functionals defined on pairs of
representation spaces in the same way as in \cite{jnsa}. In order
to introduce them, we define the following subsets on $\Lambda$:
parabola $\gamma_1:\,x_1+x_2=1$ and hyperbola
$\gamma_2=\gamma_{2,+}\cup\gamma_{2,-}$, where
$\gamma_{2,\pm}:\,x_2=\pm1$. Let $H_i$ be a subgroup of $G$, which
acts transitively on $\gamma_i$. We define ${\sf F}_1$ and ${\sf
F}_2$ as
$${\sf
F}_i:\,\,(\mathfrak{D},\hat{\mathfrak{D}})\longrightarrow\mathbb{C},\,\,(f,g)\longmapsto\int_{\gamma_i}f(x)\,g(x)\,{\rm
d}\gamma_i,$$ where ${\rm d}\gamma_i$ is a $H_i$-invariant measure
on $\gamma_i$. Let us parameterize $\gamma_1$ and $\gamma_2$ as
follow:
$$\gamma_1:\,\begin{cases}x_1=\frac12\left(1+\alpha_1^2\right),\\x_2=\frac12\left(1-\alpha_1^2\right),\\x_3=\alpha_1,\end{cases}\,\,\,\,\,\,\,
\gamma_{2,\pm}=\begin{cases}x_1=\cosh\alpha_2,\\x_2=\pm1,\\x_3=\sinh\alpha_2,\end{cases}$$
where $\alpha_1,\alpha_2\in\mathbb{R}$. Since the subgroups $H_1$
and $H_2$ consist of matrices
$$h_1(\theta_1)=\frac12\left(\begin{array}{rrr}2+\theta_1^2&\theta_1^2&2\theta_1\\
-\theta_1^2&2-\theta_1^2&-2\theta_1\\2\theta_1&2\theta_1&2\end{array}\right)$$
and $$h_2(\theta_2)=\begin{pmatrix}\cosh\theta_2&0&\sinh\theta_2\\0&1&0\\
\sinh\theta_2&0&\cosh\theta_2\end{pmatrix},$$ respectively, where
$\theta_1\in[0;2\pi)$ and $\theta_2\in\mathbb{R}$, and
\begin{gather*}T(h_1(\theta_1))[f_\lambda(\alpha_1)]=f_\lambda(\alpha_1-\theta_1),\\
T(h_2(\theta_2))[f_{\rho,\pm}(\alpha_2)]=f_{\rho,\pm}(\alpha_2-\theta_2),\end{gather*}
we have ${\rm d}\gamma_i={\rm d}\alpha_i$. It have been showed in
\cite{jnsa} that ${\sf F}_1$ and ${\sf F}_2$ coincide on pairs
$(\mathfrak{D},\mathfrak{D}^\bullet)$ such that degree of
homogeneity of $\mathfrak{D}^\bullet$ is equal to $-\sigma-1$.

\section{Matrix elements of $B_1\rightleftarrows B_2$ and $B^\bullet_1\rightleftarrows B^\bullet_2$ transformations in terms of Coulomb functions}

Let us express a function $f_\lambda\in B_1^\bullet$ as a linear
combination of vectors belonging to $B_2^\bullet$:
\begin{equation}\label{chirac}f^\bullet_\lambda(x)=\int_\mathbb{R}[c^\bullet_{\lambda,\rho,+}f^\bullet_{\rho,+}(x)+c^\bullet_{\lambda,\rho,-}f^\bullet_{\rho,-}(x)]\,{\rm
d}\rho.\end{equation} Since
\begin{equation}\label{roosevelt}f_{\rho,\pm}|_{\gamma_{2,\pm}}=f^\bullet_{\rho,\pm}|_{\gamma_{2,\pm}}=\exp({\bf
i}\rho\alpha_2)\,\,\,\,\,\,\,\,{\rm
and}\,\,\,\,\,\,\,\,f_{\rho,\pm}|_{\gamma_{2,\mp}}=f^\bullet_{\rho,\pm}|_{\gamma_{2,\mp}}=0,\end{equation}
we have
\begin{multline*}{\sf
F}_i(f^\bullet_\lambda,f_{\hat\rho,\pm})=\int_\mathbb{R}c^\bullet_{\lambda,\rho,\pm}\,{\sf
F}_2(f^\bullet_{\rho,\pm},f_{\hat\rho,\pm})\,{\rm
d}\rho\\=\int_\mathbb{R}c^\bullet_{\lambda,\rho,\pm}\,{\rm
d}\rho\,\int_\mathbb{R}\exp({\bf i}(\rho+\hat\rho)\alpha_2)\,{\rm
d}\alpha_2=2\pi\int_\mathbb{R}c^\bullet_{\lambda,\rho,\pm}\,\delta(\rho+\hat\rho)\,{\rm
d}\rho=2\pi\,c^\bullet_{\lambda,-\hat\rho,\pm},\end{multline*}
where $\delta(\rho+\hat\rho)$ is the $\hat\rho$-delayed Dirac
delta function, therefore,
$$c^\bullet_{\lambda,\rho,\pm}=\frac1{2\pi}{\sf
F}_i(f^\bullet_\lambda,f_{-\hat\rho,\pm}).$$ In the same way, if
\begin{equation}\label{obama}f^\bullet_{\rho,\pm}(x)=\int_{\mathbb{R}}c^\bullet_{\rho,\pm,\lambda}\,f^\bullet_\lambda(x)\,{\rm
d}\lambda,\end{equation} then
\begin{equation}\label{chun-doo-hwan}c^\bullet_{\rho,\pm,\lambda}=\frac1{2\pi}{\sf
F}_i(f^\bullet_{\rho,\pm},f_{-\lambda})=c_{-\lambda,-\rho,\pm}.\end{equation}
Considering that $\sigma$ is the third argument (after $\rho$ and
$\lambda$) of $c^\bullet_{\rho,\pm,\lambda}$, we derive from
\eqref{chun-doo-hwan} that
$c^\bullet_{\rho,\pm,\lambda}(\sigma)=c_{-\lambda,-\rho,\pm}(-\sigma-1)$.

\vskip 3mm

\begin{theorem} Let $\sigma\in\mathbb{R}_{-1}$ and $\lambda\ne0$,
\begin{equation*}c^\bullet_{\lambda,\rho,+}=\frac{|\Gamma(\sigma+1+{\bf
i}\rho)|}{\pi\,\lambda^{\sigma+1}}\,
\exp\left(\frac{\pi\rho}2\right)\,F_\sigma(\rho,\lambda).\end{equation*}\end{theorem}

\begin{proof} Let us use the known formula \cite[see, e.g., Entry 2.3.6(1)]{v1}
\begin{equation}\label{indira-gandhi}\int\limits_0^ax^{\alpha-1}\,(a-x)^{\beta-1}\,\exp(-px)\,{\rm
d}x={\rm
B}(\alpha,\beta)\,a^{\alpha+\beta-1}\,{}_1F_1(\alpha;\alpha+\beta;-ap),\end{equation}
which holds true for $\Re(\alpha),\Re(\beta)\in\mathbb{R}_0$, to
computing of $c_{\lambda,\rho,+}$:
\begin{multline*}c^\bullet_{\lambda,\rho,+}=\frac1{2\pi}{\sf
F}_1(f^\bullet_\lambda,f_{-\rho,+})\\=\frac1{2\pi}\int_{\mathbb{R}}\left(\frac{1-\alpha_1^2}2\right)_+^{\sigma+{\bf
i}\rho}\,\left(\frac{1+\alpha_1^2}2+\alpha_1\right)^{-{\bf
i}\rho}\,\exp({\bf i}\lambda\alpha_1)\,{\rm
d}\alpha_1\\=\frac{2^{-\sigma-1}}\pi\int\limits_{-1}^1(1-\alpha_1)^{\sigma+{\bf
i}\rho}\,(1+\alpha_1)^{\sigma-{\bf i}\rho}\,\exp({\bf
i}\lambda\alpha_1)\,{\rm d}\alpha_1\\=
\frac{2^{-\sigma-1}}\pi\,\exp(-{\bf
i}\lambda)\,\int\limits_0^2t^{\sigma-{\bf
i}\rho}\,(2-t)^{\sigma+{\bf i}\rho}\,\exp({\bf i}\lambda t)\,{\rm
d}t\\=\frac{2^\sigma}\pi\,\exp(-{\bf i}\lambda)\,{\rm
B}(\sigma+1+{\bf i}\rho,\sigma+1-{\bf
i}\rho)\,{}_1F_1(\sigma+1-{\bf i}\rho;2\sigma+2;2{\bf
i}\lambda).\end{multline*} Using here the relation (see, e.g.,
\cite[p. 290]{nu})
$$M_{\mu,\nu}(z)=z^{\nu+\frac12}\,\exp\left(-\frac{z}2\right)\,{}_1F_1\left(\nu-\mu+\frac12;2\nu+1;z\right),$$
we obtain $$c_{\lambda,\rho,+}=\frac{2^{-2\sigma-2}\,({\bf
i}\lambda)^{-\sigma-1}}\pi\,{\rm B}(\sigma+1+{\bf
i}\rho,\sigma+1-{\bf i}\rho)\,M_{{\bf i}\rho,\sigma+\frac12}(2{\bf
i}\rho).$$ Using here \eqref{mitterand} and considering the
equalities ${\rm B}(z,w)=\frac{\Gamma(z)\,\Gamma(w)}{\Gamma(z+w)}$
and $\Gamma(\overline{z})=\overline{\Gamma(z)}$, where
$\overline{z}$ is the complex conjugate of $z$, we complete the
proof.
\end{proof}

Let us note that
\begin{multline}\label{krasnoyarsk}c^\bullet_{\lambda,\rho,-}=\frac1{2\pi}{\sf
F}_1(f^\bullet_\lambda,f_{-\rho,-})\\=\frac1{2\pi}\int_{\mathbb{R}}\left(\frac{1-\alpha_1^2}2\right)_-^{\sigma+{\bf
i}\rho}\,\left(\frac{1+\alpha_1^2}2+\alpha_1\right)^{-{\bf
i}\alpha_1}\,\exp({\bf i}\lambda\alpha_1)\,{\rm
d}\alpha_1\\=\frac{2^{-\sigma-1}}\pi\left[\exp(-{\bf
i}\lambda)\int_{\mathbb{R}_0}t^{\sigma-{\bf
i}\rho}\,(t+2)^{\sigma+{\bf i}\rho}\,\exp({\bf i}\lambda t)\,{\rm
d}t\right.\\+\left.\exp({\bf
i}\lambda)\int_{\mathbb{R}_0}t^{\sigma+{\bf
i}\rho}\,(t+2)^{\sigma-{\bf i}\rho}\,\exp(-{\bf i}\lambda t)\,{\rm
d}t\right].\end{multline} Since
\begin{multline*}\left|\int_{\mathbb{R}_0}t^{\sigma\pm{\bf
i}\rho}\,(t+2)^{\sigma\pm{\bf i}\rho}\,\exp(-{\bf i}\lambda
t)\,{\rm
d}t\right|\leqslant\int_{\mathbb{R}_0}\left|t^{\sigma\pm{\bf
i}\rho}\,(t+2)^{\sigma\pm{\bf i}\rho}\,\exp(-{\bf i}\lambda
t)\right|\,{\rm d}t\\=
\int_{\mathbb{R}_0}t^\sigma\,(t+2)^\sigma\,{\rm
d}t=2^{2\sigma+1}\,{\rm B}(\sigma+1,-2\sigma-1)\end{multline*} for
$\sigma\in\left(-1,-\frac12\right)$ \cite[Entry 2.24.23]{v1}, the
both improper integrals in \eqref{krasnoyarsk} absolutely converge
for these values for $\sigma$. However, in order to represent the
matrix elements $c^\bullet_{\lambda,\rho,-}$ in terms of Coulomb
functions, we consider  the following theorem for one particular
value of $\sigma$ not belonging to the above domain.

\vskip 3mm

\begin{theorem} Let $\sigma=\frac14$, $\lambda\ne0$, and
\begin{equation*}A=\Re\Big({\bf i}^{\frac14}\,\exp\big({\bf
i}c_{-\frac14}(\rho)\big)\Big),\,\,\,\,\,\,\,\,\,B=\Im\Big({\bf
i}^{\frac14}\,\exp\big({\bf
i}c_{-\frac14}(\rho)\big)\Big).\end{equation*} Then
\begin{equation*}c^\bullet_{\lambda,\rho,-}=\frac{2\Big(A\,G_{-\frac34}(\rho,\lambda)-
B\,F_{-\frac34}(\rho,\lambda)\Big)}{\lambda^{\frac12}\,\sqrt{\cosh(2\pi\rho)}}.\end{equation*}
\end{theorem}

\begin{proof} In view of \eqref{krasnoyarsk}, we have
$$c^\bullet_{\lambda,\rho,-}=\frac1{2^{\frac34}\,\pi}\,\sum\limits_{j=0}^1
\exp((-1)^j{\bf
i}\lambda)\int_{\mathbb{R}_0}t^{-\frac14+(-1)^j{\bf
i}\rho}\,(t+2)^{-\frac14-(-1)^j{\bf i}\rho}\,\exp((-1)^j{\bf
i}\lambda t)\,{\rm d}t.$$ Using here the  formula (see, e.g.,
\cite[Entry 3.383.(6)]{gr})
\begin{equation}\label{rajiv-gandhi}\int_{\mathbb{R}_0}x^{\nu-1}\,(x+\beta)^{\frac12-\nu}\,\exp(-\mu
x)\,{\rm
d}x=\frac{2^{\nu-\frac12}}{\mu^{\frac12}}\,\Gamma(\nu)\,\exp\left(\frac{\beta\mu}2\right)
\,D_{1-2\nu}\left(\sqrt{2\beta\mu}\right),\end{equation} where
$|\arg\beta|<\pi$, $\Re(\nu)\in\mathbb{R}_0$,
$\Re(\mu)\geqslant0$, and $D_\tau$ is the parabolic cylinder
function, and considering that \cite[Entry 9.240]{gr}
$$D_{\tau}(z)=2^{\frac{2\tau+1}4}\,W_{\frac{2\tau+1}4,-\frac14}\left(\frac{z^2}2\right),$$
we obtain
\begin{equation*}c^\bullet_{\lambda,\rho,-}=\frac1{(2\lambda)^{\frac12}\,\pi}\,\sum\limits_{j=0}^1\big((-1)^j{\bf i}\big)^{\frac12}\,
\Gamma\left(\frac34+(-1)^j{\bf
i}\rho\right)\,W_{(-)^{j+1}\rho,-\frac14}\left((-1)^{j+1})2{\bf
i}\lambda\right).\end{equation*} Using here \eqref{thatcher} and
considering that
$$\left|\Gamma\left(\frac14\pm{\bf i}\rho\right)\right|=\Gamma\left(\frac14\pm{\bf i}\rho\right)\,\exp\left(\mp{\bf
i}c_{-\frac34}(\rho)\right),$$ we have
\begin{equation*}c^\bullet_{\lambda,\rho,-}=\frac{{\bf i}^{\frac12}\,\left|\Gamma\left(\frac14+{\bf i}\rho\right)\right|}{(2\lambda)^{\frac12}\,\pi}
\,\left[{\bf i}^{\frac14}\,\Gamma\left(\frac34+{\bf
i}\rho\right)\,H^+_{-\frac34}(\rho,\lambda)+ (-{\bf
i})^{\frac14}\,\Gamma\left(\frac34-{\bf
i}\rho\right)\,H^-_{-\frac34}(\rho,\lambda)\right].\end{equation*}
Considering that $$\Gamma\left(\frac14+{\bf
i}\rho\right)\,\Gamma\left(\frac34-{\bf
i}\rho\right)=\frac{\pi\,\sqrt{2}}{\cosh(\pi\rho)+{\bf
i}\sinh(\pi\rho)},$$ we obtain
\begin{equation*}c^\bullet_{\lambda,\rho,-}=\frac1{\lambda^{\frac12}\,\cosh(2\pi\rho)} \,\left[{\bf
i}^{\frac14}\,\exp\big({\bf
i}c_{-\frac14}(\rho)\big)\,H^+_{-\frac34}(\rho,\lambda)+ (-{\bf
i})^{\frac14}\,\exp\big(-{\bf
i}c_{-\frac14}(\rho)\big)\,H^-_{-\frac34}(\rho,\lambda)\right].\end{equation*}
Therefore,
$$c^\bullet_{\lambda,\rho,-}=\frac2{\lambda^{\frac12}\,\sqrt{\cosh(2\pi\rho)}}\,\Re\Big({\bf
i}^{\frac14}\,\exp\big({\bf
i}c_{-\frac14}(\rho)\big)\,H^+_{-\frac34}(\rho,\lambda)\Big).$$
\end{proof}

\section{Integrals involving products of Coulomb and modified Bessel functions, converging to Legendre
functions and related to expression of $f^\bullet _{\rho,\pm}$
with respect to basis $B^\bullet_1$}

Let $\xi\in\mathbb{R}^3$. It is clear that the function
$F_\xi(x)=(\xi_1x_1-\xi_2x_2-\xi_3x_3)^\sigma$ belongs to
$\mathfrak{D}$.

\vskip 3mm

\begin{theorem} Let \begin{equation}\label{restriction}|\xi_2|<\sqrt{\xi_1^2-\xi_3^2},\,\,\,\,\,\,\,\,\,\xi_1>\xi_3,\end{equation} and $-1<\sigma<0$.  Then

\begin{multline*}\int_{\mathbb{R}_0}\lambda^{-\frac12}\,K_{\sigma+\frac12}\left(\frac{\lambda\sqrt{
\xi_1^2-\xi_2^2-\xi_3^2}}{|\xi_1+\xi_2|}\right)\,\left[\exp\left(\frac{{\bf
i}\xi_3\lambda}{\xi_1+\xi_2}\right)\,F_{-\sigma-1}(-\rho,-\lambda)\right.\\\left.+(-1)^\sigma\,
\exp\left(-\frac{{\bf
i}\xi_3\lambda}{\xi_1+\xi_2}\right)\,F_{-\sigma-1}(-\rho,\lambda)\right]\,{\rm
d}\lambda=\frac{\pi}{2\,(\xi_1+\xi_2)^\sigma}\\\cdot\left(\frac{\xi_1-\xi_3}{\xi_1+\xi_3}\right)^{\frac{\sigma-{\bf
i}\rho}2}\,\left(\frac{|\xi_1+\xi_2|}{\sqrt{\xi_1^2-\xi_2^2}}
\right)^{\sigma+\frac12}\,\exp\left(\frac{\pi\rho}2\right)\,\Gamma(-2\sigma)\\\cdot{\rm
B}(-\sigma-{\bf i}\rho,-\sigma+{\bf i}\rho)\,|\Gamma(-\sigma+{\bf
i}\rho)|\,P_{{\bf
i}\rho-\frac12}^{\sigma+\frac12}\left(\frac{|\xi_2|}{\sqrt{\xi_1^2-\xi_3^2}}\right).\end{multline*}

\end{theorem}

\begin{proof} Let us not that \eqref{restriction} excepts the case
$\xi=(\xi_1,-\xi_1,\xi_1)$, thus numbers $\xi_1+\xi_2$,
$\xi_1-\xi_3$, $\xi_1+\xi_3$ are not equal to zero and, in
particular, the current theorem is formulated correctly. In view
of \eqref{roosevelt},
\begin{multline*}{\sf F}_i(f^\bullet_{\rho,\pm},F_\xi)={\sf
F}_2(f^\bullet_{\rho,\pm},F_\xi)=\int\limits_{\gamma_{2,+}}f^\bullet_{\rho,\pm}(x)\,F_\xi(x)\,{\rm
d}\gamma_2\\=\int_{\mathbb{R}}\exp({\bf i}\rho\alpha_2)\,
(\xi_1\cosh\alpha_2-\xi_3\sinh\alpha_2-\xi_2)^\sigma\,{\rm
d}\alpha_2\\=2^{-\sigma}\, \int_{\mathbb{R}_0}t^{-\sigma-1+{\bf
i}\rho}\,[(\xi_1-\xi_3)t^2\mp2\xi_2t+\xi_1+\xi_3]^\sigma\,{\rm
d}t,\end{multline*} where the polynomial
$(\xi_1-\xi_3)t^2\pm2\xi_2t+\xi_1+\xi_3$ doesn't have real roots
in view of \eqref{restriction}. In order to evaluate this
integral, we use the formula \cite[Entry 2.2.9.(7)]{v1}
\begin{equation}\label{paris}\int_{\mathbb{R}_0}\frac{x^{\mu-1}\,{\rm
d}x}{(ax^2+2bx+c)^\nu}=a^{-\frac\mu2}\,c^{\frac\mu2-\nu}\, {\rm
B}(\mu,2\nu-\mu)\,{}_2F_1\left(\frac\mu2,\nu-\frac\mu2;\nu+\frac12;1-\frac{b^2}{ac}\right),\end{equation}
where $a\in\mathbb{R}_0$, $b^2<ac$, $0<\Re(\mu)<2\Re(\nu)$, and
${}_2F_1$ is the Gaussian hypergeometric function. The condition
\eqref{restriction} means that the argument of this function in
\begin{multline*}{\sf
F}_i(f^\bullet_{\rho,\pm},F_\xi)=2^{-\sigma}(\xi_1-\xi_3)^{\frac{\sigma-{\bf
i}\rho}2}\,(\xi_1+\xi_3)^{-\frac{\sigma-{\bf i}\rho}2}\\\cdot {\rm
B}(-\sigma+{\bf i}\rho,-\sigma-{\bf i}\rho)\,
{}_2F_1\left(\frac{-\sigma+{\bf i}\rho}2,-\frac{\sigma+{\bf
i}\rho}2;\frac12-\sigma;1-\frac{\xi_2^2}{\xi_1^2-\xi_3^2}\right)\end{multline*}
belongs to the interval $(0;1)$, thus we use the formula
\cite[Entry 7.3.1.(41)]{v3}
$${}_2F_1(a,b;c;x)=2^{a+b-\frac12}\,x^{\frac{1-2a-2b}4}\,\Gamma\left(a+b+\frac12\right)\,
P_{a-b-\frac12}^{\frac12-a-b}\left(\sqrt{1-x}\right).$$

On the other hand, in view of \eqref{obama},
\begin{equation}\label{kennedy}{\sf
F}_i(f^\bullet_{\rho,+},F_\xi)=\int_{\mathbb{R}}c^\bullet_{\rho,+,\lambda}\,{\sf
F}_i(f^\bullet_\lambda,F_\xi)\,{\rm d}\lambda,\end{equation} where
\begin{equation*}{\sf F}_i(f^\bullet_\lambda,F_\xi)={\sf
F}_1(f^\bullet_\lambda,F_\xi)=\left(\frac{\xi_1+\xi_2}2\right)^\sigma\,\int_{\mathbb{R}}
\exp({\bf
i}\lambda\alpha_1)\,\left[\alpha_1^2-\frac{2\xi_3\alpha_1}{\xi_1+\xi_2}+\frac{\xi_1-\xi_2}{\xi_1+\xi_2}\right]^\sigma\,{\rm
d}\alpha_1.\end{equation*} In order to evaluate this integral, we
use the formula \cite[p. 202]{ob}
\begin{equation}\label{london}\int_{\mathbb{R}}\left[a^2+(x\pm b)^2\right]^{-\nu}\,
\exp({\bf i}xy)\,{\rm d}x=\frac{2\sqrt{\pi}\,\exp(\mp{\bf
i}by)}{\Gamma(\nu)}\,\left(\frac{|y|}{2a}\right)^{\nu-\frac12}\,K_{\nu-\frac12}(a|y|),\end{equation}
where $\Re(\nu)\in\mathbb{R}_0$.

Considering \eqref{kennedy} and
$\Gamma(-\sigma)\Gamma\left(\frac12-\sigma\right)=2^{2\sigma+1}\,\sqrt{\pi}\,\Gamma(-2\sigma)$
(in view of Legendre Duplication Formula), we complete the proof.
\end{proof}

Let us note that under condition of Theorem 3, the real parts of
numbers $-\sigma\pm{\bf i}\rho$ and $\sigma+1\pm{\bf i}\rho$ are
positive. It means that function $F_{-\sigma-1}$ in this theorem
can be expressed via $F_\sigma$ and $G_\sigma$: Dziecol, Yngve and
Froman derived in \cite{dyf}  the reflection formula (for {\bf
complex} $\sigma$, $\rho$ and $\lambda$)
$$F_{-\sigma-1}(\rho,\lambda)=\cos\theta\,F_\sigma(\rho,\lambda)+\sin\theta\,G_\sigma(\rho,\lambda),$$
where
$$\theta=\left(\sigma+\frac12\right)\pi+c_{-\sigma-1}(\rho)-c_\sigma(\rho),$$ which holds for $z,w\ne0$, $-\pi<\arg z,\arg w<\pi$,
 and $\ln\Gamma(z),\ln\Gamma(w)\in\mathbb{R}$ for $z,w\in\mathbb{R}_0$, where
$z=-\sigma\pm{\bf i}\rho$ and $w=\sigma+1\pm{\bf i}\rho$. In this
way, it is possible to represent $c^\bullet_{\lambda,\rho,-}$ in
Theorem 2 as a linear combinations of functions
$F_{-\frac14}(\rho,\lambda)$ and $G_{-\frac14}(\rho,\lambda)$.

\begin{theorem} For \begin{equation}\label{opposite}|\xi_2|>\sqrt{\xi_1^2-\xi_3^2}>0,\,\,\,\,\,\,\,\xi_1>\xi_3,\end{equation} and $-1<\sigma<0$,
\begin{multline*}\int_{\mathbb{R}}\lambda^{-\frac12}\,\exp\left(\frac{{\bf
i}\xi_3\lambda}{\xi_1+\xi_2}\right)\,F_{-\sigma-1}(-\rho,-\lambda)\\\cdot\Big(\big(1-\sec(\pi\sigma)\big)J_{\sigma+\frac12}(2|k|\lambda)+\tan(\pi\sigma)
J_{-\sigma-\frac12}(2|k|\lambda)\Big) \,{\rm
d}\lambda\\=\frac{(-1)^{\sigma+1}\,(\xi_1+\xi_2)^\sigma\,|\xi_1+\xi_2|^{\sigma+\frac12}\,(\xi_1+|\xi_3|)^{{\bf
i}\rho}\,\sin(\pi\sigma)\,\exp\left(\frac{3\pi\rho}2\right)}{2^\sigma\,\sqrt{\pi}\,\sqrt{\xi_2^2+\xi_3^2-\xi_1^2}\,(\xi_1^2-\xi_3^2)^{\frac{{\bf
i}\rho}2}\,\Gamma(\sigma+1)\,\Gamma(-\sigma)}\\\cdot
Q_{-\sigma-1}^{\rho{\bf
i}}\left(-\frac{\xi_2}{\sqrt{\xi_2^2+\xi_3^2-\xi_1^2}}\right),\end{multline*}
where $k=\frac{\sqrt{\xi_2^2+\xi_3^2-\xi_1^2}}{\xi_1+\xi_2}$.
\end{theorem}

\begin{proof} Let us note that \eqref{opposite} yields the inequality $|\xi_1|\ne|\xi_3|$.

In view of \eqref{roosevelt},
\begin{multline*}{\sf F}_i(f^\bullet_{\rho,+},F_\xi)={\sf
F}_2(f^\bullet_{\rho,+},F_\xi)=\int\limits_{\gamma_{2,\pm}}f^\bullet_{\rho,+}(x)\,F_\xi(x)\,{\rm
d}\gamma_2\\=\int_{\mathbb{R}}\exp({\bf i}\rho\alpha_2)\,
(\xi_1\cosh\alpha_2-\xi_3\sinh\alpha_2\mp\xi_2)^\sigma\,{\rm
d}\alpha_2\\=(\xi_1^2-\xi_3^2)^{\frac{\sigma-{\bf
i}\rho}2}\,(\xi_1+|\xi_3|)^{{\bf i}\rho}\,
\int_{\mathbb{R}}\exp({\bf i}\rho u)\,\left[\cosh
u-\frac{\xi_2}{\sqrt{\xi_1^2-\xi_3^2}}\right]^\sigma\,{\rm
d}u,\end{multline*} where
$\frac{|\xi_2|}{\sqrt{\xi_1^2-\xi_3^2}}\in\mathbb{R}_1$ according
to \eqref{opposite}. Meaning here the last integral as its
principal value and using formula \cite[Entry 2.5.48.(6)]{v1}
$$\int\limits_0^{+\infty}\frac{\cos bx\,{\rm d}x}{(a+\cosh
cx)^\nu}=\frac{\exp\left(\frac{b\pi}c\right)\,\Gamma\left(\nu-\frac{{\bf
i}b}c\right)}{c\,(a^2-1)^{\frac\nu2}\,\Gamma(\nu)}\,Q_{\nu-1}^{\frac{b{\bf
i}}c}\left(\frac{a}{\sqrt{a^2-1}}\right),$$ which is valid for
$b,\Re(c\nu)\in\mathbb{R}_0$ and $a\notin[-1;1]$, we have
\begin{multline}\label{havel}{\sf
F}_i(f^\bullet_{\rho,\pm},F_\xi)=2\exp(\rho\pi)\,(\xi_1^2-\xi_3^2)^{-\frac{{\bf
i}\rho}2}\,(|\xi_1|+|\xi_3|)^{{\bf
i}\rho}\\\cdot(\xi_2^2-\xi_3^2-\xi_1^2)^{\frac\sigma2}\,
\frac{\Gamma(-\sigma-{\bf i}\rho)}{\Gamma(-\sigma)}\,
Q_{-\sigma-1}^{\rho{\bf
i}}\left(\mp\frac{\xi_2}{\sqrt{\xi_2^2+\xi_3^2-\xi_1^2}}\right).\end{multline}

On the other hand, \eqref{obama} yields \eqref{kennedy}, where the
polynomial
$\alpha_1^2-\frac{2\xi_3\alpha_1}{\xi_1+\xi_2}+\frac{\xi_1-\xi_2}{\xi_1+\xi_2}$
has two different real roots $\frac{\xi_3}{\xi_1+\xi_2}\pm k$.
Using the substitution $\alpha_1=t+\frac{\xi_3}{\xi_1+\xi_2}$, we
have
\begin{equation*}{\sf F}_i(f^\bullet_\lambda,F_\xi)={\sf
F}_1(f^\bullet_\lambda,F_\xi)=\left(\frac{\xi_1+\xi_2}2\right)^\sigma\,\exp\frac{{\bf
i}\xi_3\lambda}{\xi_1+\xi_2}\,\int_{\mathbb{R}}\exp({\bf i}\lambda
t)\,(t^2-4k^2)^\sigma\,{\rm d}t,\end{equation*} where (according
to \cite[Entry 2.3.5.3]{v1})
\begin{equation*}\int\limits_{-2|k|}^{2|k|}\exp({\bf i}\lambda
t)\,(4k^2-t^2)^\sigma\,{\rm
d}t=\left(\frac{4|k|}\lambda\right)^{\sigma+\frac12}\,\sqrt{\pi}\,\Gamma(\sigma+1)\,
J_{\sigma+\frac12}(2\lambda|k|)
\end{equation*} and (according
to \cite[Entry 2.3.5.5]{v1})
\begin{multline*}\sum\limits_{j=0}^1\,\int\limits_{\mathbb{R}_{2|k|}}\exp((-1)^j{\bf
i}\lambda t)\,(t^2-4k^2)^\sigma\,{\rm
d}t\\=4^\sigma\,\sqrt{\pi}\,{\bf
i}\,\left(\frac{\lambda}{|k|}\right)^{-\sigma-\frac12}\,\Gamma(\sigma+1)\,\Big(H^{(1)}_{-\sigma-\frac12}
(2\lambda|k|)-H^{(2)}_{-\sigma-\frac12}
(2\lambda|k|)\Big),\end{multline*} where
$H^{(1)}_{-\sigma-\frac12}$ and $H^{(2)}_{-\sigma-\frac12}$ are
Hankel functions of the first and second kind, respectively. To
finish the proof we use the identity $$H^{(1)}_{-\sigma-\frac12}
(2\lambda|k|)-H^{(2)}_{-\sigma-\frac12} (2\lambda|k|)=2{\bf
i}\big(\sec(\pi\sigma)\,J_{\sigma+\frac12}(2\lambda|k|)-\tanh(\pi\sigma)\,J_{-\sigma-\frac12}(2\lambda|k|)\big).$$

\end{proof}

\section{Integrals involving products of Coulomb and Legendre functions, converging to modified Bessel functions and related to expression of $f^\bullet _\lambda$
with respect to basis $B^\bullet_2$}

The results obtained in Theorems 3 and 4 be may  be characterised
as formulae, on the one hand, for exponential Fourier and, for the
second hand, for Mellin transform of Coulomb functions. In
addition, Theorem 3 is being formula for $K$-transform and Theorem
4 is being formula for sum of Hankel transforms \cite{be2} of
Coulomb functions. All these formulas have been derived from the
expression of the function belonging to 'hyperbolic' basis with
respect to 'parabolic' basis. Choosing the opposite direction, in
this section we derive one formula for
 generalized index Mehler--Fock transform \cite{y}.

\begin{theorem} For \eqref{restriction} and $\lambda\in\mathbb{R}_0$, \begin{multline*}
\int\limits_{\mathbb{R}}\left[\pi^{-1}\,\left|\Gamma\left(\frac34+{\bf
i}\rho\right)\right|\,\exp\left(\frac{\pi\rho}2\right)\,F_{-\frac14}(\rho,\lambda)+\lambda^{\frac14}\,
\frac{2\Big(A\,G_{-\frac34}(\rho,\lambda)-
B\,F_{-\frac34}(\rho,\lambda)\Big)}{\sqrt{\cosh(2\pi\rho)}}\right]\\
(\xi_1-\xi_3)^{-\frac18-{\bf i}\rho}\,(\xi_1+\xi_3)^{-\frac18+{\bf
i}\rho}\,{\rm B}\left(\frac14+{\bf i}\rho, \frac14-{\bf
i}\rho\right)\,P_{{\bf
i}\rho-\frac12}^{\frac14}\left(\frac{|\xi_2|}{\sqrt{\xi_1^2-\xi_3^2}}\right)\,{\rm
d}\rho\\=2^{\frac34}\,\left(\frac\lambda\pi\right)^{\frac12}\,(\xi_1^2-\xi_3^2)^{\frac18}\,
K_{\frac14}\left(\frac{\lambda\,\sqrt{\xi_1^2-\xi_2^2-\xi_3^2}}{|\xi_1+\xi_2|}\right).
\end{multline*} \end{theorem}

\begin{proof} From \eqref{chirac} we have $${\sf F}_i(f^\bullet_\lambda,F_\xi)=\int_\mathbb{R}[c^\bullet_{\lambda,\rho,+}
{\sf F}_j(f^\bullet_{\rho,+},F_\xi)+c^\bullet_{\lambda,\rho,-}{\sf
F}_k(f^\bullet_{\rho,-},F_\xi)]\,{\rm d}\rho,$$ where
$i,j,k\in\{1,2\}$. Choosing here $i=1$ and $j=k=2$, we calculate
${\sf F}_1(f^\bullet_\lambda,F_\xi)$ and ${\sf
F}_2(f^\bullet_{\rho,+},F_\xi)$ according to formulas
\eqref{london} and \eqref{paris}, respectively, and use Theorems 1
and 2.
\end{proof}

By using the formula (see, for example, \cite{v3})
$$P_\nu^\mu(z)=\sqrt{\frac2\pi}\,\frac{{\bf i}\,\exp({\bf
i}\nu\pi)}{(z^2-1)^{\frac14}}\,Q_{-\mu-\frac12}^{-\nu-\frac12}\left(\frac{z}{\sqrt{z^2-1}}\right),$$
which is valid for $\Re(z)\in\mathbb{R}_0$, it is possible to
obtain the result, which is similar to Theorem 5, for
\eqref{opposite}.

\section{The relationship with the Poisson transform in $\mathfrak{D}^\bullet$}

Let us note that the integrals ${\sf
F}_i(f^\bullet_\lambda,F_\xi)$ and ${\sf
F}_i(f^\bullet_{\rho,\pm},F_\xi)$, which do not depend on the
choice of integration contour $\gamma_i$, are the particular cases
of the so-called Poisson transform \cite[Section
10.3.1]{vk}$$\mathcal{P}[f](y)=\int_{\gamma_i}f(x)\,(x_1y_1-x_2y_2-x_3y_3)^\sigma\,{\rm
d}\gamma_i,$$ where $f\in\mathfrak{D}^\bullet$ and the point $y$
belongs to the hyperboloid $\Upsilon:\,y_1^2-y_2^2-y_3^2=1$. The
image of this integral transform consists of $\sigma$-homogeneous
functions which are defined on $\Upsilon$ and belong to the kernel
of the following '$(1,2)$-Laplace operator':
$$\square=\frac{\partial^2}{\partial y_1^2}-\frac{\partial^2}{\partial
y_2^2}-\frac{\partial^2}{\partial y_3^2}.$$ The Poisson transform
intertwines the representation $T^\bullet$ and the representation
defined by the shifts
$\mathcal{P}[f](y)\longmapsto\mathcal{P}[f](g^{-1}y)$. Thus,
computing ${\sf F}_i(f^\bullet_\lambda,F_\xi)$ and ${\sf
F}_i(f^\bullet_{\rho,\pm},F_\xi)$ for the case
\begin{equation}\label{march}\xi=y=(\cosh\alpha_3,\sinh\alpha_3\cos\beta_3,\sinh\alpha_3\sin\beta_3),\end{equation}
where $\alpha_3\in\mathbb{R}$ and $\beta_3\in(-\pi,\pi)$, which
satisfies the condition \eqref{restriction}, we obtain the values
of the Poisson transform with the kernel $F_y$ of the basis
functions $f^\bullet_\lambda$ and $f^\bullet_{\rho,\pm}$. For
example, we have
\begin{multline*}\mathcal{P}[f^\bullet_\lambda](y)=\sqrt{\frac{2\pi}{\cosh\alpha_3+\sinh\alpha_3\cos\beta_3}}\,
\frac{|\lambda|^{-\sigma-\frac12}}{\Gamma(-\sigma)}\\\cdot\exp\left(\frac{{\bf
i}\lambda\sinh\alpha_3\sin\beta_3}{\cosh\alpha_3+\sinh\alpha_3\cos\beta_3}\right)\,
K_{\sigma+\frac12}\left(\frac{|\lambda|}{\cosh\alpha_3+\sinh\alpha_3\cos\beta_3}\right).\end{multline*}
The function $F_y$ as a function $F_x(y)$ defined on $\Upsilon$ is
also being the kernel of the so-called Gelfand--Graev integral
transform. The applications of this transform to generalizations
of Funk--Hecke theorem had been considered in \cite{vk} and
\cite{s}.

We note also that choosing $\beta=0$ in \eqref{march}, we obtain
from Theorem 3 the following integral representation of Legendre
function: \begin{multline*}P_{{\bf
i}\rho-\frac12}^{\sigma+\frac12}(|\tanh\alpha_3|)=\frac{2\,\exp\left(-\frac{\alpha_3+\pi\rho}2\right)}
{\pi\,\Gamma(-2\sigma)\,{\rm B}(-\sigma-{\bf i}\rho,-\sigma+{\bf
i}\rho)\,|\Gamma(-\sigma+{\bf i}\rho)|}\\
\cdot\int_{\mathbb{R}_0}\frac{F_{-\sigma-1}(-\rho,\-\lambda)+(-1)^\sigma\,F_{-\sigma-1}(-\rho,\lambda)}{\sqrt{\lambda}}\,
K_{\sigma+1}\left(\frac\lambda{\exp\alpha_3}\right)\,{\rm
d}\lambda .\end{multline*} A more general representation for
$P_{{\bf i}\rho-\frac12}^{\sigma+\frac12}(|\tanh\alpha_3|)$ can be
obtained by using the following parametrization of $\Upsilon$:
$$y=(\cosh\alpha_3\cosh\beta_3,\sinh\alpha_3,\cosh\alpha_3\sinh\beta_3).$$

\bigskip

\renewcommand{\refname}{

\end{document}